\documentclass[12pt]{amsart}
\usepackage{amsmath,amsthm,amssymb}
\usepackage{mathrsfs}
\usepackage{xcolor}
\font\teneufm=eufm10
\font\seveneufm=eufm7
\font\fiveeufm=eufm5
\newfam\frakturfam
\textfont\frakturfam=\teneufm
\scriptfont\frakturfam=\seveneufm
\scriptscriptfont\frakturfam=\fiveeufm

\theoremstyle{definition}
\newtheorem{definition}{Definition}[section]
\newtheorem{remark}[definition]{Remark}
\newtheorem{example}[definition]{Example}

\theoremstyle{plain}
\newtheorem{lemma}[definition]{Lemma}
\newtheorem{prp}[definition]{Proposition}
\newtheorem{corollary}[definition]{Corollary}
\newtheorem{theorem}[definition]{Theorem}

\newcommand{\Zinb}{\operatorname{Zinb}}
\newcommand{\BinZinb}{\operatorname{BinZinb}}
\newcommand{\spn}{\operatorname{span}}
\numberwithin{equation}{section}

\begin{document}

\title[Binary Zinbiel algebras]{Nilpotency and nil subvarieties of binary Zinbiel algebras}

\author{Nurken Smadyarov}

\address{Institute of Mathematics and Mathematical Modeling, Almaty, Kazakhstan and Farabi Atyrau International School, Atyrau, Kazakhstan}

\email{nurken.smadyar@gmail.com}

\date{}
\subjclass[2020]{17A30, 17A50}
\keywords{Binary Zinbiel algebra, Zinbiel defect, free algebra, solvability,
nilpotency, nil variety}

\begin{abstract}
An algebra is called binary Zinbiel if every subalgebra generated by
at most two elements is a Zinbiel algebra. Over a field of
characteristic zero, we prove that every finite-dimensional binary Zinbiel algebra is nilpotent. We also show that a subvariety of binary Zinbiel algebras
is nil if and only if it does not contain the variety of all Zinbiel
algebras. We further prove an analogue of It\^o's theorem for binary Zinbiel algebras.
\end{abstract}

\maketitle
\markboth{BINARY ZINBIEL ALGEBRAS}{BINARY ZINBIEL ALGEBRAS}

\section{Introduction}\label{sec:introduction}

Let $\mathbb{F}$ be a field of characteristic zero. A variety of algebras over
$\mathbb{F}$ is a class defined by polynomial identities in noncommuting,
nonassociating variables. Given a variety $\mathcal V$, an algebra is
called a {\it mono $\mathcal V$-algebra} if each of its one-generated
subalgebras belongs to $\mathcal V$, and a {\it binary $\mathcal V$-algebra}
if each of its two-generated subalgebras belongs to $\mathcal V$.
These constructions lead to the study of identities and structural
properties that can be detected on subalgebras with a prescribed number
of generators. For example, Artin's theorem identifies binary
associative algebras with alternative algebras; see the discussion and
references in \cite{IMS}.

An algebra $A$ is called a {\it right Zinbiel} algebra if it satisfies
\begin{equation}\label{eq:zinbiel}
a(bc)=(ab)c+(ba)c.
\end{equation}
Zinbiel algebras, also called dual Leibniz algebras, were introduced by
Loday in connection with Leibniz cohomology and the Koszul duality of
operads \cite{Loday}. We use the right convention throughout the paper.
The left convention, used in \cite{DT,NU,Towers}, is obtained by passing
to the opposite algebra.

Ismailov, Mashurov and Smadyarov \cite{IMS} determined defining
identities for mono and binary Zinbiel algebras. In particular, an
algebra over $\mathbb{F}$ is binary Zinbiel if and only if the trilinear map
\begin{equation}\label{eq:defect}
\phi(a,b,c)=a(bc)-(ab)c-(ba)c
\end{equation}
satisfies
\begin{equation}\label{eq:binary-identities}
\phi(a,b,a)=0,
\qquad
\phi(a,a,b)=0.
\end{equation}
Equivalently, the map $\phi$ is alternating. The class of binary Zinbiel algebras properly contains
the class of Zinbiel algebras. The question whether every
finite-dimensional binary Zinbiel algebra is nilpotent was raised in
\cite[Remark~1.5]{IMS}.

Nilpotency for ordinary Zinbiel algebras has been studied from both
structural and identity-theoretic points of view. Dzhumadil'daev and
Tulenbaev \cite{DT} proved an analogue of the Nagata--Higman theorem:
in characteristic zero, a Zinbiel algebra of bounded nil index is
nilpotent. They also proved solvability of finite-dimensional Zinbiel
algebras over algebraically closed fields and nilpotency in
characteristic zero. Towers \cite{Towers} subsequently proved
nilpotency of finite-dimensional Zinbiel algebras over arbitrary
fields by studying minimal ideals. On the level of varieties,
Naurazbekova and Umirbaev \cite{NU} proved that every proper subvariety
of the variety of Zinbiel algebras in characteristic zero is
nilpotent.

Another structural question concerns algebras that can be written as
the sum of two abelian subalgebras. In 1955, It\^o~\cite{Ito} proved
that a group which is the product of two abelian subgroups is
metabelian. Analogues of this theorem have been studied for Lie
algebras~\cite{BK}, associative algebras~\cite{Militaru},
Leibniz algebras~\cite{AM}, Novikov algebras~\cite{SZ},
Poisson algebras~\cite{BCZ}, and non-commutative generalized Poisson algebras~\cite{DIM}.

In this paper, we study defect identities, nilpotency and subvarieties
of binary Zinbiel algebras. We first prove two identities governing
products involving the defect $\phi$. They yield a spanning set for
the free binary Zinbiel algebra consisting of left-normed products in
the free generators and their alternating three-variable defects
(Lemma~\ref{lem:spanning}).
After establishing solvability by the method of
Dzhumadil'daev and Tulenbaev, we adapt the minimal-ideal argument of
Towers and prove that every finite-dimensional binary Zinbiel algebra
over $\mathbb{F}$ is nilpotent (Theorem~\ref{thm:binary-zinbiel-nilpotency}). Finally, we show that a subvariety $\mathcal W$ of binary
Zinbiel algebras is nil if and only if it does not contain the whole
variety of Zinbiel algebras (Proposition~\ref{prop:bounded-nil}).
We also prove that if a binary Zinbiel algebra is the sum of two
abelian subalgebras, then it is metabelian
(Theorem~\ref{thm:ito-binary-zinbiel}). This result requires no
finite-dimensionality assumption.

The paper is organized as follows. Section~\ref{sec:defects} contains
the defect identities and the spanning result. Solvability in finite
dimension is established in Section~\ref{sec:solvability}.
Section~\ref{sec:nilpotency} treats minimal ideals and proves the
nilpotency theorem. Section~\ref{sec:subvarieties} is devoted to nil
subvarieties.
Section~\ref{sec:ito} establishes the analogue of It\^o's theorem.

\section{Preliminaries and defect identities}\label{sec:defects}

Throughout the paper, $\mathbb{F}$ is a field of characteristic zero, and
algebras are $\mathbb{F}$-vector spaces with a bilinear multiplication; no unit
is assumed. All Zinbiel identities are written in the right
convention~\eqref{eq:zinbiel}. The varieties of right Zinbiel and
right binary Zinbiel algebras are denoted by $\Zinb$ and $\BinZinb$,
respectively. By a two-generated subalgebra we mean one generated by at most two
elements.

For subspaces $U,V$ of an algebra $A$, put
$$
UV=\spn_\mathbb{F}\{uv:u\in U,\ v\in V\}.
$$
A left ideal $B$ satisfies $AB\subseteq B$; a two-sided ideal also
satisfies $BA\subseteq B$. The derived series and the left-normed
powers of $A$ are defined by
\begin{equation}\label{eq:series}
\begin{aligned}
A^{(0)}&=A, & A^{(k+1)}&=A^{(k)}A^{(k)} &&(k\ge0),\\
A^1&=A, & A^{k+1}&=A^kA &&(k\ge1).
\end{aligned}
\end{equation}
The algebra $A$ is \emph{solvable} if $A^{(r)}=0$ for some $r$.
It is \emph{nilpotent} if there is an integer $N$ such that every
product of at least $N$ elements, with arbitrary parentheses, is zero.
For an element $a\in A$, write
\begin{equation}\label{eq:element-powers}
a^{[1]}=a,\qquad a^{[k+1]}=a^{[k]}a.
\end{equation}
The algebra is \emph{nil} if for every $a\in A$ there is an integer
$n(a)$ such that $a^{[n(a)]}=0$. A variety is nil if all of its
members are nil. It has uniformly bounded nil index if it satisfies
$x^{[N]}=0$ for one fixed $N$.

Let $A$ be a binary Zinbiel algebra. Linearizing
\eqref{eq:binary-identities} gives
$$
\phi(a,b,c)+\phi(c,b,a)=0,
\qquad
\phi(a,c,b)+\phi(c,a,b)=0.
$$
Thus, for every $\sigma\in S_3$,
\begin{equation}\label{iden-1}
\phi(x_1,x_2,x_3)=\operatorname{sgn}(\sigma)
\phi(x_{\sigma(1)},x_{\sigma(2)},x_{\sigma(3)}).
\end{equation}

\begin{lemma}\label{lem:defect-identities}
Every binary Zinbiel algebra satisfies the identities
\begin{equation}\label{iden-2}\phi(ab,c,d)=a\phi(b,c,d).\end{equation}
\begin{equation}\label{iden-3}\phi(\phi(a,b,c),d,e)=0.\end{equation}
\end{lemma}

\begin{proof}
Set
$$
q(a,b,c)
=
\phi(a,b,c)+\phi(b,a,c),
$$
and $$
\begin{aligned}
S(a,b,c,d)=&
2\phi(ab,c,d)
+2\phi(ba,c,d)
+\phi(a,b,cd)
\\
&-2\phi(a,c,bd)
-2a\phi(b,c,d)
-c\phi(a,b,d).
\end{aligned}
$$
A direct calculation gives 
$$
S(a,b,c,d)=
q(ab,c,d)+q(ba,c,d)-q(a,c,bd)-a\,q(b,c,d).$$
Indeed,
$$q(ab,c,d)
=(ab)(cd)+c((ab)d)-2((ab)c)d-2(c(ab))d,$$
$$q(ba,c,d)=
(ba)(cd)+c((ba)d)-2((ba)c)d-2(c(ba))d,$$
$$q(a,c,bd)=a(c(bd))+c(a(bd))-2(ac)(bd)-2(ca)(bd),$$
$$aq(b,c,d)=
a(b(cd))+a(c(bd))-2a((bc)d)-2a((cb)d).$$
This linear combination is precisely the expansion of $S(a,b,c,d)$.
Since $q=0$ by alternation, we obtain $S(a,b,c,d)=0$.

We now prove \eqref{iden-2}.
Since $S=0$, every permutation of its variables is also an identity.
Consider
$$
\begin{aligned}
0={}&
2S(a,b,c,d)
+2S(a,b,d,c)
+4S(b,d,c,a)
\\
&-4S(c,a,b,d)
-2S(c,a,d,b)
-S(c,d,a,b)
\\
&+4S(d,a,b,c)
+4S(d,b,a,c).
\end{aligned}
$$
Substituting the definition of $S$ and using the alternating property
of $\phi$, all terms cancel except
$$
3\phi(ab,c,d)-3a\phi(b,c,d).
$$
Hence
$$
3\phi(ab,c,d)-3a\phi(b,c,d)=0.
$$
Since $char\mathbb{F}=0$,
$$
\phi(ab,c,d)=a\phi(b,c,d).
$$
We next prove \eqref{iden-3}. 
First, by alternation and \eqref{iden-2},
$$
\phi(x,y,zt)
=
\phi(zt,x,y)
=
z\phi(t,x,y)
=
z\phi(x,y,t).
$$
Hence
\begin{equation}\label{eq:phi-third-product}
\phi(x,y,zt)=z\phi(x,y,t).
\end{equation}

Now calculate $\phi(bc,d, ae)$ in two ways. On the one hand,
$$
\phi(bc,d,ae)
=
b\phi(c,d,ae)
=
b\bigl(a\phi(c,d,e)\bigr).
$$
On the other hand, by \eqref{eq:phi-third-product},
$$
\phi(bc,d,ae)
=
a\phi(bc,d,e)
=
a\bigl(b\phi(c,d,e)\bigr).
$$
Therefore
\begin{equation}\label{eq:commuting-on-defects}
a\bigl(b\phi(c,d,e)\bigr)
=b\bigl(a\phi(c,d,e)\bigr).
\end{equation}

Put $u=\phi(c,d,e)$. By \eqref{eq:commuting-on-defects},
we have $a(bu)=b(au)$. Hence
$$
0=q(a,b,u)
 =2\bigl(a(bu)-(ab)u-(ba)u\bigr)
 =2\phi(a,b,u).
$$
Since $\operatorname{char}\mathbb{F}=0$, it follows that
$$
\phi(a,b,\phi(c,d,e))=0.
$$
Finally, by alternation,
$$
\phi(\phi(a,b,c),d,e)
=\phi(d,e,\phi(a,b,c))
=0.
$$
\end{proof}

For a binary Zinbiel algebra $A$, let
\begin{equation}\label{eq:defect-space}
D=\spn_\mathbb{F}\{\phi(a,b,c):a,b,c\in A\}.
\end{equation}
Lemma~\ref{lem:defect-identities} and alternation imply
\begin{equation}\label{eq:defect-space-properties}
AD\subseteq D,
\qquad
\phi(D,A,A)=\phi(A,D,A)=\phi(A,A,D)=0.
\end{equation}

Let $\mathcal F(X)$ be the free binary Zinbiel algebra on
$$
X=\{x_1,x_2,\ldots\},
$$
and put
$$
Y=\{\phi(x_i,x_j,x_k):i<j<k\},
\qquad Z=X\cup Y.
$$

\begin{lemma}\label{lem:spanning}
Every element of the free binary Zinbiel algebra $\mathcal F(X)$ is a linear
combination of left-normed products
$$
(\cdots((z_1z_2)z_3)\cdots)z_m,
\qquad z_1,\ldots,z_m\in Z.
$$
\end{lemma}

\begin{proof}
We use induction on the degree of monomials in the generators $X$.

The statement is clear in degrees $1$ and $2$. For degree $3$, by the
definition of $\phi$,
$$
a(bc)=(ab)c+(ba)c+\phi(a,b,c),
$$
and hence the statement holds.

Assume that the statement is true for all monomials of degree less
than $n$. Let $w$ be a monomial of degree $n$. Write
$$
w=uv,
\qquad
\deg u<n,\quad \deg v<n.
$$
By the induction assumption, it is sufficient to consider
$$
u=(\cdots((z_1z_2)z_3)\cdots)z_p, \quad v=(\cdots((t_1t_2)t_3)\cdots)t_q,
$$
where $z_i,t_j\in Z$.

We prove by an auxiliary induction on the number $p+q$ of
$Z$-factors that $uv$ is a linear combination of left-normed products
in $Z$.

If $q=1$, then $uv$ is already left-normed. Let $q\geq2$ and write
$$
v=v't_q.
$$
By the definition of $\phi$,
$$
u(v't_q)
=
(uv')t_q+(v'u)t_q+\phi(u,v',t_q).
$$
By the induction assumption on $p+q$, the first two terms belong to
the required linear span.

It remains to consider $\phi(u,v',t_q)$. If $p\geq2$, write
$$
u=u'z_p.
$$
By \eqref{iden-2},
$$
\phi(u,v',t_q)
=
u'\phi(z_p,v',t_q).
$$
If $q\geq3$, write
$$
v'=v''t_{q-1}.
$$
Using \eqref{iden-1} and \eqref{iden-2}, we obtain
$$
\phi(z_p,v',t_q)
=
-\phi(v',z_p,t_q)
=
-v''\phi(t_{q-1},z_p,t_q).
$$
Therefore,
$$
\phi(u,v',t_q)
=
-u'\bigl(v''\phi(t_{q-1},z_p,t_q)\bigr).
$$

If $r,s,t\in X$, then $\phi(r,s,t)=0$ when two arguments
coincide; otherwise, by alternation, $\phi(r,s,t)=\pm y$ for some
$y\in Y$. If one of the arguments belongs to $Y$, then
\eqref{iden-1} and \eqref{iden-3} give
$$
\phi(r,s,t)=0.
$$
Thus, for $p\geq2$ and $q\geq3$, the term $\phi(u,v',t_q)$ is either
zero or is, up to sign, of the form
$$
u'(v''y),
\qquad y\in Y.
$$
The latter involves fewer than $p+q$ elements of $Z$, so the induction
assumption applies again.

If $p=1$ and $q\geq3$, then
$$
\phi(u,v',t_q)
=
-v''\phi(t_{q-1},u,t_q),
$$
which is either zero or a scalar multiple of a left-normed product
in $Z$. If $p\geq2$ and
$q=2$, then
$$
\phi(u,t_1,t_2)
=
u'\phi(z_p,t_1,t_2),
$$
and the same conclusion holds. Finally, if $p=1$ and $q=2$, then
$\phi(u,t_1,t_2)$ is zero or equals $\pm y$ for some $y\in Y$.

Hence $uv$ is a linear combination of left-normed products in $Z$.
The induction is complete.
\end{proof}

\section{Solvability in finite dimension}\label{sec:solvability}

Throughout this section, $A$ is a finite-dimensional right binary
Zinbiel algebra over a field $\mathbb{F}$. The defect
space $D$ is defined by
\eqref{eq:defect-space}.

Following Dzhumadil'daev and Tulenbaev~\cite[Section~2]{DT}, we use a
maximal-left-annihilator argument. For $a\in A$, put
$$
C(a)=\{x\in A:xa=0\}.
$$

\begin{lemma}\label{lem:annihilated-defect}
If $D\neq0$, then there exists $0\neq a\in D$ such that $Aa=0$.
\end{lemma}

\begin{proof}
For $a\in D$ and $x,b\in A$, the vanishing in
\eqref{eq:defect-space-properties} gives
$$
x(ba)=(xb+bx)a=b(xa).
$$
It follows that
$$
C(a)\subseteq C(ba).
$$

Choose $0\neq a_0\in D$ for which $\dim C(a_0)$ is maximal.
Since $ba_0\in D$, the preceding inclusion and maximality give
$$
C(a_0)=C(ba_0)
\qquad\text{whenever }ba_0\neq0.
$$
Suppose that $C(a_0)\neq A$. Choose $b_0\in A$ such that
$b_0a_0\neq0$, and define
$$
b_{k+1}=b_kb_k
\qquad(k\ge0).
$$
We claim that $b_ka_0\neq0$ for all $k$. The assertion holds
for $k=0$. If $b_ka_0\neq0$ but $b_{k+1}a_0=0$, the identity
$\phi(b_k,b_k,a_0)=0$ gives
$$
b_k(b_ka_0)=2(b_kb_k)a_0=2b_{k+1}a_0=0.
$$
Hence
$$
b_k\in C(b_ka_0)=C(a_0),
$$
contrary to $b_ka_0\neq0$. This proves the claim.

Let $B=\langle b_0\rangle$ be the one-generated subalgebra of $A$.
By the definition of binary Zinbiel algebras, $B$ is a right
Zinbiel algebra. It is finite-dimensional, so the ordinary
Zinbiel nilpotency theorem \cite[Theorem~2.4]{Towers},
applied to $B^{\mathrm{op}}$, implies that $B$ is nilpotent.
Thus sufficiently long products in $B$ vanish with arbitrary
parentheses. Since $b_k$ is a product of $2^k$ copies of $b_0$,
we obtain $b_k=0$ for all sufficiently large $k$.
This contradicts $b_ka_0\neq0$. Therefore $C(a_0)=A$.
\end{proof}

\begin{lemma}\label{lem:proper-ideal}
If $\dim A>1$, then $A$ has a nonzero proper two-sided ideal.
\end{lemma}

\begin{proof}
If $D=0$, then $A$ is a right Zinbiel algebra and is nilpotent
by \cite[Theorem~2.4]{Towers}, applied to $A^{\mathrm{op}}$.

Suppose that $D\neq0$. By Lemma~\ref{lem:annihilated-defect},
choose $0\neq a\in D$ such that $Aa=0$, and put
$$
I=aA=\{ay:y\in A\}.
$$
Since every defect with an argument in $D$ vanishes,
for all $x,y\in A$ we have
$$
(ay)x=a(yx)-(ya)x=a(yx)\in I
$$
and
$$
x(ay)=(xa+ax)y=(ax)y=a(xy)\in I.
$$
Thus $I$ is a two-sided ideal.

The nonzero vector $a$ lies in the kernel of
$$
L_a:A\longrightarrow A,\qquad L_a(y)=ay,
$$
because $aa=0$. Therefore
$$
\dim I=\dim\operatorname{im}L_a\le\dim A-1.
$$
If $I\neq0$, this is the required ideal. If $I=0$, then
$aA=Aa=0$, and $\mathbb{F}a$ is a nonzero proper two-sided ideal.
\end{proof}

\begin{theorem}\label{thm:finite-dimensional-solvability}
Every finite-dimensional right binary Zinbiel algebra over
a field of characteristic zero is solvable.
\end{theorem}

\begin{proof}
We argue by induction on $n=\dim A$. The case $n=0$ is immediate.
If $n=1$, write $A=\mathbb{F}e$ and $ee=\alpha e$. The identity
$e(ee)=2(ee)e$ yields $\alpha^2e=2\alpha^2e$, so $\alpha=0$.
Thus $A^{(1)}=0$.

Let $n>1$. By Lemma~\ref{lem:proper-ideal}, $A$ has a nonzero
proper two-sided ideal $I$. Both $I$ and $A/I$ are binary
Zinbiel algebras of dimension less than $n$. By induction,
there exist $r,s\ge1$ such that
$$
(A/I)^{(r)}=0,
\qquad I^{(s)}=0.
$$
For the quotient homomorphism $\pi:A\to A/I$, induction on $k$
gives
$$
\pi\bigl(A^{(k)}\bigr)=(A/I)^{(k)},
$$
because $\pi(UV)=\pi(U)\pi(V)$ for all subspaces $U,V\subseteq A$.
Hence $A^{(r)}\subseteq I$. A further induction gives
$$
A^{(r+k)}\subseteq I^{(k)}
\qquad(k\ge0).
$$
In particular, $A^{(r+s)}\subseteq I^{(s)}=0$.
Therefore $A$ is solvable.
\end{proof}

\begin{remark}\label{rem:mono-not-solvable}
The solvability conclusion of
Theorem~\ref{thm:finite-dimensional-solvability} does not extend to
finite-dimensional mono Zinbiel algebras. Note that every
anticommutative algebra over a field of characteristic different from
$2$ is mono Zinbiel. For every element $x$, anticommutativity gives
$x^2=0$, so the subalgebra generated by $x$ has zero multiplication
and is therefore a Zinbiel algebra. In particular, consider the three-dimensional simple Lie algebra
$A=\mathfrak{sl}_2(\mathbb F)$, with multiplication $xy=[x,y]$. Since $A^2=A$, it follows that
$A^{(n)}=A$ for every $n\geq 0$. Thus $A$ is not solvable.
\end{remark}

\section{Minimal ideals and nilpotency}\label{sec:nilpotency}

Unless stated otherwise, $A$ is a finite-dimensional right binary
Zinbiel algebra over a field $\mathbb{F}$.
By Theorem~\ref{thm:finite-dimensional-solvability},
$A$ is solvable. We adapt the minimal-ideal argument of
Towers~\cite[Section~2]{Towers}, using the defect identities from
Section~\ref{sec:defects}.

\begin{lemma}\label{lem:towers-ideal}
Let $B$ be a left ideal of $A$ such that
$$
\phi(B,A,A)=0.
$$
Then $BA$ is a two-sided ideal of $A$.
\end{lemma}

\begin{proof}
Since $B$ is a left ideal, $AB\subseteq B$.
For $b\in B$ and $x,y\in A$, the equality
$\phi(b,x,y)=0$ gives
$$
(bx)y=b(xy)-(xb)y.
$$
Both terms on the right belong to $BA$, since $xb\in B$.
Therefore
$$
(BA)A\subseteq BA.
$$

By alternation, $\phi(x,b,y)=0$, and hence
$$
x(by)=(xb)y+(bx)y.
$$
The first term belongs to $BA$ because $xb\in B$,
and the second belongs to $BA$ by the inclusion just proved.
Thus
$$
A(BA)\subseteq BA.
$$
Consequently, $BA$ is a two-sided ideal.
\end{proof}

\begin{prp}\label{prop:towers-minimal-left-ideal}
Let $M$ be a minimal nonzero two-sided ideal of $A$.
Suppose that $B$ is a minimal nonzero left ideal of $A$
such that
$$
B\subseteq M,
\qquad
\phi(B,A,A)=0.
$$
Then $M=B$.
\end{prp}

\begin{proof}
By Lemma~\ref{lem:towers-ideal}, $BA$ is a two-sided ideal.
Since
$$
BA\subseteq MA\subseteq M,
$$
minimality of $M$ implies
$$
BA=0
\qquad\text{or}\qquad
BA=M.
$$
If $BA=0$, then $B$ is a two-sided ideal, because
$AB\subseteq B$. Hence $B=M$.

Suppose that $BA=M$.

For $x,y\in A$ and $b\in B$, the vanishing of the defect
and its alternation give
$$
x(yb)=(xy+yx)b=y(xb).
$$
It follows that $A^2B$ is a left ideal: for $s\in A^2$,
$$
x(sb)=s(xb)\in A^2B.
$$
Moreover, $A^2B\subseteq B$. Thus minimality of $B$ gives
$$
A^2B=0
\qquad\text{or}\qquad
A^2B=B.
$$

Suppose that $A^2B=B$. If $B\subseteq A^{(k)}B$, then
$$
B
\subseteq A^{(k)}B
\subseteq A^{(k)}\bigl(A^{(k)}B\bigr)
\subseteq A^{(k+1)}B.
$$
The last inclusion follows from
$$
u(vb)=(uv+vu)b
\qquad(u,v\in A^{(k)},\ b\in B).
$$
Starting with $B=A^{(1)}B$, induction gives
$$
B\subseteq A^{(k)}B
\qquad(k\ge1).
$$
This contradicts solvability of $A$ and $B\neq0$.
Therefore
$$
A^2B=0.
$$

The equality $x(yb)=y(xb)$ also shows that $AB$
is a left ideal contained in $B$.
If $AB=B$, then
$$
B=A(AB)\subseteq A^2B=0,
$$
a contradiction. Consequently,
$$
AB=0.
$$

We next show that $BA^2$ is a two-sided ideal.
For $b\in B$, $s\in A^2$, and $x\in A$, we have
$$
(bs)x=b(sx)-(sb)x=b(sx)\in BA^2,
$$
because $sb=0$.
Similarly,
$$
x(bs)=(xb+bx)s=(bx)s=b(xs)\in BA^2.
$$
Thus $BA^2$ is a two-sided ideal contained in $M$.
Minimality of $M$ gives
$$
BA^2=0
\qquad\text{or}\qquad
BA^2=M.
$$

Suppose that $BA^2=M$. Then $B\subseteq BA^{(1)}$.
Since $AB=0$ and $\phi(B,A,A)=0$, we have
$$
(bu)v=b(uv)
\qquad(b\in B,\ u,v\in A).
$$
Therefore, if $B\subseteq BA^{(k)}$, then
$$
B
\subseteq BA^{(k)}
\subseteq \bigl(BA^{(k)}\bigr)A^{(k)}
\subseteq BA^{(k+1)}.
$$
Induction and solvability of $A$ yield $B=0$,
a contradiction. Hence
$$
BA^2=0.
$$

Finally, $M=BA$ implies
$$
MA=(BA)A\subseteq BA^2=0
$$
and
$$
AM=A(BA)\subseteq (AB)A+(BA)A=0.
$$
Every subspace of $M$ is therefore a two-sided ideal of $A$.
Minimality of $M$ forces $\dim M=1$.
Since $0\neq B\subseteq M$, we conclude that $B=M$.
\end{proof}

\begin{corollary}\label{cor:towers-minimal-ideal-annihilator}
Every minimal nonzero two-sided ideal $M$ of $A$ satisfies
$$
AM=MA=0,
\qquad
\dim M=1.
$$
\end{corollary}

\begin{proof}
We first show that $M$ is a minimal nonzero left ideal
and that
$$
\phi(M,A,A)=0.
$$

Suppose first that $M\cap D\neq0$.
Since $M$ is an ideal and $AD\subseteq D$,
the subspace $M\cap D$ is a nonzero left ideal.
Choose a minimal nonzero left ideal $B\subseteq M\cap D$.
The inclusion $B\subseteq D$ gives
$$
\phi(B,A,A)=0.
$$
By Proposition~\ref{prop:towers-minimal-left-ideal},
$M=B$. Thus $M$ is a minimal nonzero left ideal,
and $\phi(M,A,A)=0$.

Now suppose that $M\cap D=0$.
For $m\in M$ and $x,y\in A$, the ideal property gives
$$
\phi(m,x,y)=m(xy)-(mx)y-(xm)y\in M.
$$
By definition, this element also belongs to $D$.
Therefore
$$
\phi(M,A,A)\subseteq M\cap D=0.
$$
Choose a minimal nonzero left ideal $B\subseteq M$.
Proposition~\ref{prop:towers-minimal-left-ideal}
again gives $M=B$.

Thus, in both cases, $M$ is a minimal nonzero left ideal
and all defects with an argument in $M$ vanish.

For $x,y\in A$ and $m\in M$, we have
$$
x(ym)=(xy+yx)m=y(xm).
$$
Hence $AM$ is a left ideal contained in $M$, so
$$
AM=0
\qquad\text{or}\qquad
AM=M.
$$
Suppose that $AM=M$.
If $A^{(k)}M=M$, then
$$
M
=A^{(k)}\bigl(A^{(k)}M\bigr)
\subseteq A^{(k+1)}M
\subseteq M.
$$
Starting with $A^{(0)}M=M$, induction gives
$$
A^{(k)}M=M
\qquad(k\ge0),
$$
contradicting solvability of $A$.
Therefore
$$
AM=0.
$$

For $x,y\in A$ and $m\in M$, we now obtain
$$
x(my)=(xm+mx)y=(mx)y=m(xy).
$$
Thus $MA$ is a left ideal contained in $M$.
By minimality,
$$
MA=0
\qquad\text{or}\qquad
MA=M.
$$
Suppose that $MA=M$.
The identities $AM=0$ and $\phi(M,A,A)=0$ imply
$$
(mu)v=m(uv)
\qquad(m\in M,\ u,v\in A).
$$
If $MA^{(k)}=M$, it follows that
$$
M
=\bigl(MA^{(k)}\bigr)A^{(k)}
\subseteq MA^{(k+1)}
\subseteq M.
$$
Induction again contradicts solvability.
Consequently,
$$
MA=0.
$$

Since $AM=MA=0$, every subspace of $M$ is a two-sided
ideal of $A$. Minimality of $M$ therefore gives
$$
\dim M=1.
$$
\end{proof}

\begin{theorem}\label{thm:binary-zinbiel-nilpotency}
Every finite-dimensional right binary Zinbiel algebra
over a field of characteristic zero is nilpotent.
\end{theorem}

\begin{proof}
We argue by induction on $n=\dim A$.
The assertion is immediate if $n=0$.
If $n=1$, the calculation in the proof of
Theorem~\ref{thm:finite-dimensional-solvability} gives $A^2=0$.

Let $n>1$, and assume that the assertion holds for algebras
of smaller dimension.
Choose a minimal nonzero two-sided ideal $M$ of $A$.
By Corollary~\ref{cor:towers-minimal-ideal-annihilator},
$$
AM=MA=0,
\qquad
\dim M=1.
$$
In particular, $M$ is proper.

The quotient $A/M$ is a binary Zinbiel algebra of smaller
dimension. By induction, it is nilpotent.
Using the left-normed powers defined in \eqref{eq:series},
we obtain
$$
A^r\subseteq M
$$
for some positive integer $r$. Consequently,
$$
A^{r+1}=A^rA\subseteq MA=0.
$$

By Lemma~\ref{lem:spanning}, the vanishing of all sufficiently
long left-normed products implies the vanishing of all
sufficiently long products, regardless of the placement
of parentheses. Thus $A$ is nilpotent, completing the induction.
\end{proof}

\section{Nil subvarieties}\label{sec:subvarieties}

In this section, $\mathbb{F}$ is any field of characteristic zero;
no finite-dimensionality assumption is imposed.

By Lemma~\ref{lem:spanning}, for any element $a$ of a binary
Zinbiel algebra, every product of $n$ copies of $a$, regardless
of the placement of parentheses, is a scalar multiple of
the left-normed power $a^{[n]}$. Thus, the use of left-normed
powers in the definition of a nil algebra entails no loss
of generality.

We first recall the integral example from
\cite[Section~1]{DT}, written in the right convention.

\begin{example}\label{ex:integral}
The polynomial algebra
$P=\mathbb{F}[x]$, equipped with the multiplication
$$
(f*g)(x)=g(x)\int_0^x f(t)\,dt,
$$
where integration is understood formally, is a right Zinbiel algebra.

The left-normed powers of the constant polynomial $1$ satisfy
\begin{equation}\label{eq:integral-powers}
1^{[n]}=\frac{x^{n-1}}{(n-1)!}\neq0
\qquad(n\ge1).
\end{equation}
Indeed, the formula holds for $n=1$, and the induction step follows from
$$
1^{[n+1]}
=1^{[n]}*1
=\int_0^x \frac{t^{n-1}}{(n-1)!}\,dt
=\frac{x^n}{n!}.
$$

Therefore the constant polynomial $1$ is not nilpotent,
and $P$ is not nil.
\end{example}

\begin{prp}\label{prop:bounded-nil}
Let $\mathcal W$ be a subvariety of binary Zinbiel algebras
over a field of characteristic zero.
The following conditions are equivalent:
\begin{enumerate}
\item $\Zinb\nsubseteq\mathcal W$.
\item Every algebra in $\mathcal W$ is nil.
\end{enumerate}
\end{prp}

\begin{proof}
Assume first that $\Zinb\nsubseteq\mathcal W$. We must show that
every algebra $A\in\mathcal W$ is nil. Consider the intersection
$$
\mathcal U=\mathcal W\cap\Zinb.
$$
The assumption $\Zinb\nsubseteq\mathcal W$
implies that $\mathcal U$ is a proper subvariety of
$\Zinb$. By the theorem of Naurazbekova and
Umirbaev~\cite[Theorem~2]{NU}, $\mathcal U$ is nilpotent. Hence there exists
an integer $N$ such that every product of $N$ elements
in every algebra belonging to $\mathcal U$, with arbitrary
parentheses, is zero.

Let $A\in\mathcal W$ and $a\in A$. Since $A$ is binary
Zinbiel, its one-generated subalgebra
$$
B=\langle a\rangle
$$
is a Zinbiel algebra. Moreover, varieties are closed under
taking subalgebras, so $B\in\mathcal W$. Therefore,
$$
B\in\mathcal W\cap\Zinb=\mathcal U.
$$

By the choice of $N$, every product of $N$ copies of $a$,
with arbitrary parentheses, vanishes. In particular,
$$
a^{[N]}=0.
$$
Thus every element of $A$ is nilpotent, so $A$ is nil. Since $A$ was an arbitrary member of $\mathcal W$, the variety $\mathcal W$ is nil.

Conversely, suppose that $\mathcal W$ is nil.
By Example~\ref{ex:integral}, the integral Zinbiel algebra $P$
is not nil, since \eqref{eq:integral-powers} gives
$1^{[n]}\neq0$ for every $n\ge1$.
If $\Zinb\subseteq\mathcal W$, then $P\in\mathcal W$,
a contradiction. Therefore $\Zinb\nsubseteq\mathcal W$.
\end{proof}

\section{It\^o's theorem}\label{sec:ito}

An algebra $C$ is called \emph{abelian} if $C^2=0$. An algebra $Z$
is called \emph{metabelian} if
$$
(Z^2)^2=0,
$$
or, equivalently, if $(xy)(zt)=0$ for all $x,y,z,t\in Z$.
In terms of the derived series~\eqref{eq:series}, this means
$Z^{(2)}=0$. No finite-dimensionality assumption is imposed in
this section.

We use the Zinbiel defect $\phi$ defined in~\eqref{eq:defect}
and its alternating property~\eqref{iden-1}.

\begin{theorem}[It\^o's theorem for binary Zinbiel algebras]
\label{thm:ito-binary-zinbiel}
Let $Z$ be a right binary Zinbiel algebra over $\mathbb F$.
Suppose that
$$
Z=A+B,
\qquad
A^2=B^2=0,
$$
where $A$ and $B$ are subalgebras of $Z$. Then $Z$ is metabelian:
$
(Z^2)^2=0.
$
\end{theorem}

\begin{proof}
Let $C$ be either $A$ or $B$. We first show that, for
$u,v\in C$ and $x\in Z$,
\begin{equation}\label{eq:ito-same-summand}
\begin{aligned}
u(vx)&=-v(ux),\\
u(xv)&=-2u(vx),\\
(ux+xu)v&=v(ux).
\end{aligned}
\end{equation}
Since $uv=vu=0$, the definition of the defect gives
$$
\phi(u,v,x)=u(vx),
\qquad
\phi(v,u,x)=v(ux).
$$
The first equality in~\eqref{eq:ito-same-summand} follows from
alternation. Invariance under cyclic permutations also gives
$$
-(xu+ux)v
=\phi(x,u,v)
=\phi(u,v,x)
=u(vx).
$$
Together with the first equality, this proves the third equality.
Finally,
$$
u(xv)-(ux+xu)v
=\phi(u,x,v)
=-\phi(u,v,x)
=-u(vx).
$$
Substituting $(ux+xu)v=-u(vx)$ yields the second equality.

Let $p\in A$, $q\in B$, and $z\in Z$. Choose a decomposition
$z=z_A+z_B$, where $z_A\in A$ and $z_B\in B$. Since $B^2=0$,
the second equality in~\eqref{eq:ito-same-summand} gives
$$
p(qz)=p(qz_A)=-2p(z_Aq)=-2p(zq).
$$
Thus
\begin{equation}\label{eq:ito-different-summands}
p(qz)=-2p(zq)
\end{equation}
whenever $p$ and $q$ belong to different summands. The same
argument applies with $A$ and $B$ interchanged.

Fix $a,c\in A$ and $b,d\in B$, and put
$$
r=c((ab)d),
\qquad
s=c((ba)d).
$$
Applying~\eqref{eq:ito-different-summands} three times, we obtain
$$
-2r=c(d(ab))=-2c(d(ba))=4s.
$$
Consequently,
\begin{equation}\label{eq:ito-r-two-s}
r=-2s.
\end{equation}
On the other hand, the third equality
in~\eqref{eq:ito-same-summand}, applied to $b,d\in B$, gives
$$
(ba+ab)d=d(ba).
$$
Multiplying on the left by $c$ and
using~\eqref{eq:ito-different-summands}, we obtain
$$
r+s=c(d(ba))=-2s,
$$
so
\begin{equation}\label{eq:ito-r-three-s}
r=-3s.
\end{equation}
Comparing~\eqref{eq:ito-r-two-s} and~\eqref{eq:ito-r-three-s}
gives $s=0$, and then $r=0$. Hence
\begin{equation}\label{eq:ito-fourfold-AB}
c((ab)d)=c((ba)d)=0
\end{equation}
for all $a,c\in A$ and $b,d\in B$. Interchanging $A$ and $B$
also gives
\begin{equation}\label{eq:ito-fourfold-BA}
d((ab)c)=d((ba)c)=0.
\end{equation}

Let $u=ab$ or $u=ba$, and choose a decomposition
$$
u=p+q,
\qquad
p\in A,\quad q\in B.
$$
For $c\in A$ and $d\in B$, identities
\eqref{eq:ito-same-summand}, \eqref{eq:ito-fourfold-AB},
and~\eqref{eq:ito-fourfold-BA} imply
$$
\begin{aligned}
p(cd)&=-c(pd)=-c(ud)=0,\\
q(cd)&=-2q(dc)=2d(qc)=2d(uc)=0.
\end{aligned}
$$
Here $pd=ud$ because $qd=0$, and $qc=uc$ because $pc=0$.
Therefore
$$
u(cd)=p(cd)+q(cd)=0.
$$
Moreover, the second equality in~\eqref{eq:ito-same-summand}
gives $p(dc)=-2p(cd)=0$, while its first equality gives
$$
q(dc)=-d(qc)=-d(uc)=0.
$$
Thus $u(dc)=0$ as well.

Finally, $A^2=B^2=0$ implies
$$
Z^2=(A+B)(A+B)=AB+BA.
$$
Every product of two spanning elements of $AB+BA$ is zero by
the preceding calculation. By bilinearity, $
(Z^2)^2=0$.
\end{proof}

\begin{remark}\label{rem:ito-counterexamples}
The analogue of It\^o's theorem does not hold for
mono Zinbiel algebras. Let $Z$ be the three-dimensional
algebra over $\mathbb F$ with basis $e_1,e_2,e_3$ and nonzero
products
$$
e_1e_3=e_3,\qquad e_2e_3=e_1,\qquad
e_3e_1=-e_3,\qquad e_3e_2=-e_1.
$$
All other products of basis elements are zero.

The algebra $Z$ is anticommutative, so $x^2=0$ for every $x\in Z$.
Consequently, every one-generated subalgebra
$\langle x\rangle=\mathbb F x$ has zero multiplication and is
therefore a Zinbiel algebra. Thus $Z$ is mono Zinbiel.

Set
$
A=\operatorname{span}_{\mathbb F}\{e_1,e_2\}$
and $B=\mathbb F e_3$. 
Then $Z=A\oplus B$ and $A^2=B^2=0$.
Nevertheless,
$$
(e_2e_3)(e_1e_3)=e_1e_3=e_3\neq0,
$$
so $Z$ is not metabelian. More explicitly,
$
Z^2=\operatorname{span}_{\mathbb F}\{e_1,e_3\}$
and $(Z^2)^2=\mathbb F e_3\neq0$.
Hence the binary assumption in
Theorem~\ref{thm:ito-binary-zinbiel} cannot be replaced by
the mono assumption.
\end{remark}
\section*{Acknowledgments}

The author is supported by grant AP26199089 from the Ministry of Science and Higher Education of the Republic of Kazakhstan.  The author thanks Nurlan Ismailov and Farukh Mashurov for their interest in this work and for valuable discussions.

\end{document}